\documentclass[12pt]{article}

\usepackage{amsmath,amsthm,amsfonts,amssymb, color,xcolor,subcaption,graphicx} 

\newtheorem{theorem}{Theorem}[section]

\newtheorem{lemma}[theorem]{Lemma}
\newtheorem{definition}[theorem]{Definition}
\newtheorem{remark}[theorem]{Remark}

\def\cB{\mathcal{B}}

\def\cE{\mathcal{E}}
\def\cF{\mathcal{F}}

\def\cN{\mathcal{N}}

\def\cS{\mathcal{S}}

\def\bE{\mathbb{E}}

\def\bN{\mathbb{N}}
\def\bP{\mathbb{P}}
\def\bR{\mathbb{R}}

\def\k{\kappa}

\newcommand{\fm}{\mathfrak{m}}

\newcommand{\lipf}{h}
\newcommand{\kerf}{f}

\begin{document}

\title{Almost sure CLT for hyperbolic Anderson model \\
with L\'evy colored noise}

\author{Raluca M. Balan\footnote{Corresponding author. University of Ottawa, Department of Mathematics and Statistics, 150 Louis Pasteur Private, Ottawa, Ontario, K1N 6N5, Canada. E-mail address: rbalan@uottawa.ca.} \footnote{Research supported by a grant from the Natural Sciences and Engineering Research Council of Canada.}
\and
Hanniel E. Kouam\'e\footnote{University of Ottawa, Department of Mathematics and Statistics, 150 Louis Pasteur Private, Ottawa, Ontario, K1N 6N5, Canada. E-mail address: mkoua060@uottawa.ca.}
\and
William D. Stephenson\footnote{University of Ottawa, Department of Mathematics and Statistics, 150 Louis Pasteur Private, Ottawa, Ontario, K1N 6N5, Canada. E-mail address: wstep051@uottawa.ca.}
}

\date{February 27, 2026}
\maketitle

\begin{abstract}
\noindent In this note, we prove the Almost Sure Central Limit Theorem (ASCLT) for the spatial integral of the solution of the hyperbolic Anderson model driven by the L\'evy colored noise introduced in \cite{B15}. For this, we
use the central limit theorem for the normalized spatial integral, and an estimate for the Malliavin derivative of the solution, both derived in the recent preprint \cite{BS26}. We assume that the spatial correlation kernel of the noise is either integrable, or it is given by the Riesz kernel.
\end{abstract}

\noindent {\em MSC 2020:} Primary 60H15; Secondary 60G60, 60G51

\vspace{1mm}

\noindent {\em Keywords:} stochastic partial differential equations, random fields, Malliavin calculus, almost sure central limit theorem, Poisson random measure, L\'evy white noise

\pagebreak

\section{Introduction}

In this article, we consider the hyperbolic Anderson model:
\begin{align}
\label{HAM}
	\begin{cases}
		\dfrac{\partial^2 u}{\partial^2 t} (t,x)
		=  \dfrac{\partial^2 u}{\partial x^2} (t,x)+u(t,x) \dot{X}(t,x), \
		t>0, \ x \in \bR, \\
		u(0,x) = 1, \dfrac{\partial u}{\partial t} (0,x)=0, \quad x \in \bR,
	\end{cases}
\end{align}
where $X$ is the {\em L\'evy colored noise} introduced in \cite{B15}. 
The recent article \cite{BJ25} shows that equation \eqref{HAM} has a unique solution, and gives some properties of the moments of the solution. This analysis was continued in article \cite{BS26} which studies the asymptotic Gaussian fluctuations as $R \to \infty$ of the {\em spatial average} of the (centered) solution:
\begin{equation}
\label{def-F}
F_R(t)=\int_{-R}^R \big( u(t,x)-1\big)dx.
\end{equation}
 
The goal of this short note is to complement this analysis and show that the normalized process $\{\widetilde{F}_R(t)\}_{R>0}$ satisfies the {\em almost sure central limit theorem (ASCLT)}, where
\begin{equation}
\label{def-Ft}
\widetilde{F}_R(t)=\frac{F_R(t)}{\sigma_R(t)} \quad \mbox{and} \quad
\sigma_{R}^2(t)={\rm Var}\big(F_R(t)\big).
\end{equation}

We recall that the noise $X$ is defined by
\[
X_t(\varphi)=L_t(\varphi *\k), \quad \mbox{for all} \ t>0,\varphi \in \cS(\bR),
\]
where $\cS(\bR)$ is the set of rapidly decreasing functions on $\bR$, $\k:\bR \to [0,\infty]$ is a suitable kernel (which satisfies Assumptions A1 and A2 of \cite{BJ25}), and $L$ is the {\em L\'evy white noise} given by
\[
L(B)=\int_{B \times \bR_0} z \widehat{N}(dt,dx,dz) \quad \mbox{for all} \quad B \in \cB_b(\bR_+ \times \bR).
\]
with $\cB_b(\bR_+ \times \bR)$ being the class of bounded Borel subsets of $\bR_{+}\times \bR$.

Here $N$ is a Poisson random measure on $Z=\bR_{+}\times \bR \times \bR_0$ with intensity 
\[
\fm(dt,dx,dz)=dt dx \nu(dz),
\] 
defined on a complete probability space $(\Omega,\cF,\bP)$, and  $\widehat{N}(A)=N(A)-\fm(A)$ is the compensated version of $N$.
The space $\bR_0=\bR \verb2\2 \{0\}$ is equipped with the distance $d(x,y)=|x^{-1}-y^{-1}|$, so that the bounded subsets of $\bR_0$ are those that are bounded away from 0.
We assume that the measure $\nu$ satisfies the following condition:
\[
m_2:=\int_{\bR_0}|z|^2 \nu(dz)<\infty.
\]
which guarantees that the noise $L$ has finite variance.

\medskip

A predictable process $\{u(t,x);t\geq 0,x\in \bR\}$ is called a {\em (mild) solution} of \eqref{HAM}
if it satisfies the integral equation:
\[
u(t,x)=1+\int_0^t \int_{\bR}G_{t-s}(x-y)u(s,y)X(ds,dy),
\]
where the stochastic integral is interpreted in the It\^o sense.

A random field $\{\Phi(t,x);t\geq 0,x\in \bR\}$ is {\em predictable} if it is measurable with respect to the predictable $\sigma$-field on $\Omega \times \bR_{+} \times \bR$, which is the minimal $\sigma$-field with respect to which all elementary processes are measurable. An {\em elementary process} is a linear combination of processes of the form
\[
\Phi(t,x)=Y 1_{(a,b]}(t) 1_{A}(x)
\]
where  $0\leq a<b$, $A \in \cB_b(\bR)$ and $Y$ is $\cF_a$-measurable. 
 Here $(\cF_t)_{t\geq 0}$ is the filtration induced by $N$, i.e.
\[
\cF_t=\sigma(\{N([0,s] \times A \times B); s \in [0,t],A\in \cB_b(\bR), B \in \cB_b(\bR_0)\}) \vee \cN,
\]
where $\cN$ is the class of $\bP$-null sets.

\medskip

We recall the classical ASCLT. Let $\{X_k\}_{k\geq 1}$ be a sequence of i.i.d. random variables defined on a probability space $(\Omega,\cF,\bP)$, with $\bE[X_1]=0$ and $\bE[X_1^2]=1$. Define $S_N=\sum_{k=1}^{N}X_k$. Then, for $\bP$-almost all $\omega \in \Omega$,
\[
\frac{1}{\log N}\sum_{k=1}^{N}\frac{1}{k}\delta_{S_k(\omega)/\sqrt{k}} \Longrightarrow \gamma \quad \mbox{as}  \quad N \to \infty,
\]
where $\gamma$ is the standard normal distribution and $\Longrightarrow$ denotes convergence in distribution. This result was stated by Paul L\'evy on page 270 of \cite{levy54} without a proof. The proof and various extensions were given few decades later by several authors: see \cite{brosamler88,lacey-philipp90,schatte88}. 
In the present paper, we will prove a continuous-time version of the ASCLT, which is well-suited for the study of SPDEs, and is motivated by L\'evy's result. We refer the reader to Section 1.2 of \cite{BXZ} for further comments on ASCLT in this continuous-time framework.

\medskip

We recall the following definition from \cite{BXZ}. 

\begin{definition}
{\rm
a) A sequence $\{ F_k\}_{k \in \bN}$ of random variables defined on a probability space $(\Omega,\cF,\bP)$ {\em satisfies
the ASCLT} if for $\bP$-almost every $\omega \in \Omega$,
\[
\mu_N^{\omega}:= \frac{1}{\log N}\sum_{k=1}^{N}\frac{1}{k}\delta_{F_k(\omega)} \Longrightarrow \gamma \quad \mbox{as}  \quad N \to \infty.
\]

b) A family  $\{ F_{\theta}\}_{\theta >0}$ of random variables {\em satisfies
the  ASCLT} if for $\bP$-almost every
 $\omega\in\Omega$, the map $\theta \mapsto F_{\theta}(\omega)$ is  measurable, and
\[
\nu_T^{\omega}:= 
\frac{1}{\log T}\int_1^T  \frac{1}{\theta}\delta_{F_{\theta}(\omega)}   dy
\Longrightarrow \gamma \quad \mbox{as}  \quad T \to \infty.
\]
}
\end{definition}

Recall that the {\em Riesz kernel of order $\alpha \in (0,1)$} is given by:
\[
R_{1,\alpha}(x)=C_{1,\alpha}|x|^{-(1-\alpha)} \quad \mbox{where} \quad C_{1,\alpha}=\pi^{-1/2}2^{-\alpha}\frac{\Gamma(\frac{1-\alpha}{2})}{\Gamma(\frac{\alpha}{2})}.
\]

We are now ready to state the main result of this paper.

\begin{theorem}
\label{main}
Assume that $\kappa \in L^1(\bR)$ or $\kappa=R_{1,\alpha/2}$ for some $\alpha \in (0,1)$. Then the process
$\{\widetilde{F}_{\theta}(t)\}_{\theta >0}$ defined by \eqref{def-Ft} satisfies the ASCLT, for any $t>0$.
\end{theorem}

Article \cite{BXZ} presents two methods for proving the ASCLT for
 the spatial integral of the solution of (HAM) driven by the L\'evy white noise $L$: \\
 (i) a direct proof based on Clark-Ocone formula; \\
 (ii) a longer proof based on the Ibragimov-Lifshitz criterion and the recent second-order Poincar\'e inequality due to \cite{trauthwein25}. 
 The second method is also used in \cite{xia-zheng25} for proving the ASCLT for the spatial integral of the solution of (HAM) or (PAM), driven by a Gaussian noise, which is colored in space and time. 
 
 \medskip

The goal of the present note is to present an adaptation of method (i) which works for the L\'evy colored noise introduced in \cite{BJ25}. At the end, in Remark \ref{Gaussian}, we show that this method can also be used for the spatially-colored Gaussian noise considered in \cite{dalang99}, providing an alternative (shorter) proof for the result of \cite{xia-zheng25}, when the noise is white in time.
 
 \section{Proof of Theorem \ref{main}}

We will use the following results, which were proved in \cite{BS26}: 
\begin{description}
\item[(i)] $\sigma_R^2(t) \sim K(t) R^{\beta}$ as $R \to \infty$, where $K(t)>0$ is a constant that depends on $t$, and
    \begin{equation}
\label{def-beta}
\beta:=\left\{
\begin{array}{ll} 1 & \mbox{if $\k \in L^1(\bR)$} \\
\alpha+1 & \mbox{if $\k=R_{1,\alpha/2}$ for some $\alpha \in (0,1)$ }
\end{array} \right.
\end{equation}

\item[(ii)] $\widetilde{F}_R(t) \stackrel{d}{\to} N(0,1)$ as $R\to \infty$

\item[(iii)] for any $0\leq r\leq t$, $x \in \bR$, $y \in \bR$ and $z \in \bR_0$,
\[
\|D_{r,y,z}u(t,x)\|_2 \leq C_t |z|\int_{\bR}G_{t-r}(x-y')\k(y-y')dy',
\]
where $C_t>0$ is a constant that depends on $t$.

\end{description}

We fix $t>0$. We denote $\widetilde{F}_{\theta}=\widetilde{F}_{\theta}(t)$ and $\sigma_{\theta}=\sigma_{\theta}(t)$.
We have to prove that for any bounded Lipschitz function $\lipf$ on $\bR$, 
\[
\frac{1}{\log T}\int_1^T\frac{1}{\theta} \lipf(\widetilde{F}_{\theta})d\theta \stackrel{a.s.}{\longrightarrow} \int_{\bR}\lipf(x)\gamma(dx) \quad \mbox{as} \quad T \to \infty.
\] 
In view of item (ii) above, it is enough to show that
\begin{equation}
\label{claim}
\frac{1}{\log T}\int_1^T\frac{1}{\theta} \Big( \lipf(\widetilde{F}_{\theta})-\bE[\lipf(\widetilde{F}_{\theta})] \Big) d\theta \stackrel{a.s.}{\longrightarrow} 0 \quad \mbox{as} \quad T \to \infty.
\end{equation}

Let $H_{\theta}=\lipf(\widetilde{F}_{\theta})-\bE[\lipf(\widetilde{F}_{\theta})]$. We use the Clark-Ocone formula with respect to $\widehat{N}$,
\[
H_{\theta}=\int_0^t \int_{\bR}\int_{\bR_0}\bE[D_{r,y,z}H_{\theta}|\cF_r]\widehat{N}(dr,dy,dz).
\]

Similarly to relation (3.2) of \cite{BXZ}, we have: 
\begin{equation}
\label{H-tw}
|\bE[H_{\theta} H_{w}]| \leq \frac{{\rm Lip}^2(\lipf)}{\sigma_{\theta} \sigma_{w}} \int_0^t \int_{\bR}\int_{\bR_0} \|D_{r,y,z}F_{\theta}\|_2 \|D_{r,y,z}F_{w}\|_2 \nu(dz)dydr,
\end{equation}
where ${\rm Lip}(\lipf)$ is the Lipschitz constant of $\lipf$, and $\|\cdot\|_2$ is the norm in $L^2(\Omega)$.

We denote
\[
\varphi_{t,\theta}(r,y)=\int_{-\theta}^{\theta}G_{t-r}(x-y)dx \quad \mbox{and} \quad I_{\theta}(y)=\varphi_{t,\theta}(0,y).
\]

By Minkowski's inequality, item (iii) above, and Fubini's theorem
\begin{align*}
\|D_{r,y,z}F_{\theta}\|_2 & \leq \int_{-\theta}^{\theta}\|D_{r,y,z}u(t,x)\|_2 dx \leq C_t |z|
\int_{-\theta}^{\theta} \int_{\bR} G_{t-r}(x-y')\k(y-y')dy' dx\\
&= C_t |z| \int_{\bR}\varphi_{t,\theta}(r,y')\k(y-y')dy'.
\end{align*}

Inserting this estimate into \eqref{H-tw}, we obtain:
\begin{align*}
& |\bE[H_{\theta}H_w]| \leq \\
& \quad \frac{{\rm Lip}^2(\lipf)}{\sigma_{\theta} \sigma_{w}} C_t^2 m_2 \int_0^t \int_{\bR} \left( \int_{\bR}\varphi_{t,\theta}(r,y')\k(y-y')dy' \right) 
\left( \int_{\bR}\varphi_{t,\theta}(r,y'')\k(y-y'')dy'' \right) dydr.
\end{align*}
Note that $\varphi_{t,\theta}(r,y)\leq I_{\theta}(y)$ for any $r\in [0,t]$ and $y \in \bR$, since $G_{t-r}(\cdot)\leq G_t(\cdot)$. Hence,

\begin{align*}
|\bE[H_{\theta}H_w]| & \leq \frac{{\rm Lip}^2(\lipf)}{\sigma_{\theta} \sigma_{w}} C_t^2 m_2 t\int_{\bR^{3}} I_{\theta}(y') I_{w}(y'') \k(y-y') \k(y-y'') dy'dy'' dy\\
&=\frac{{\rm Lip}^2(\lipf)}{\sigma_{\theta} \sigma_{w}} C_t^2 m_2 t\int_{\bR^2} I_{\theta}(y') I_w(y'')\kerf(y'-y'')dy'dy''.
\end{align*}

using the fact that $\kerf=\kappa * \widetilde{\kappa}$. Using the notation
\[
A_{\theta,w}:=\int_{\bR}\int_{\bR} I_{\theta}(y) I_w(z)\kerf(y-z)dydz,
\]
we have:
\begin{equation}
\label{H-A}
\big|\bE[H_{\theta}H_w]\big| \leq {\rm Lip}^2(\lipf) C_t^2 m_2 t \frac{A_{\theta,w}}{\sigma_{\theta} \sigma_{w}}.
\end{equation}
To estimate $A_{\theta,w}/(\sigma_{\theta}\sigma_w)$, we consider separately the two cases. We assume that $\theta<w$.

\medskip

{\em Case 1.} Assume that $\kappa \in L^1(\bR)$. By Young's inequality, $\kerf \in L^1(\bR)$. We use the fact that $I_{w}(z)\leq \int_{\bR}G_t(x-z)dx=t$. It follows that

\begin{align*}
A_{\theta,w} & \leq  t\|\kerf\|_{L^1(\bR)}
\int_{\bR}I_{\theta}(y)  dy= t \|\kerf\|_{L^1(\bR)}
\int_{-\theta}^{\theta} \int_{\bR} G_t(x-y) dy dx
=2\theta t \|\kerf\|_{L^1(\bR)}.
\end{align*}

Since in this case $\sigma_{\theta}^2 \sim K(t) \theta$, we infer that
\[
\frac{A_{\theta,w}}{\sigma_{\theta}\sigma_{w}}  \leq  C_{t,\kerf} \left(\frac{\theta}{w}\right)^{1/2},
\]
where $C_{t,\kerf}>0$ is a constant that depends on $(t,\kerf)$. Relation \eqref{claim} now follows by \eqref{H-A} and Lemma 3.1 of \cite{BXZ}.

\medskip

{\em Case 2.} Assume that $\k=R_{1,\alpha/2}$ for some $\alpha \in (0,1)$. Then $\kerf=\kappa * \tilde{\kappa}=R_{1,\alpha}$ and
\[
A_{\theta,w}=C_{1,\alpha}\int_{\bR} \int_{\bR}I_{\theta}(y) I_{w}(z)|y-z|^{-(1-\alpha)}dydz.
\]

To calculate $A_{\theta,w}$, we use a method based on Fourier transforms. 
Note that the Fourier transform of $\kerf$ in $\cS'(\bR)$ is $g(\xi)=|\xi|^{-\alpha}$, i.e. $(\kerf,\cF \phi)=(g,\phi)$ for any $\phi \in \cS(\bR)$. By the Fourier inversion theorem,
$(\kerf,\varphi)=\frac{1}{2\pi}(g,\cF \varphi)$ for any $\varphi \in \cS(\bR)$, that is
\[
\int_{\bR}|x|^{-(1-\alpha)} \varphi(x)dx=\frac{1}{2\pi C_{1,\alpha} }\int_{\bR}\cF \varphi(\xi)|\xi|^{-\alpha}d\xi.
\]
Hence, for any $\varphi,\psi \in \cS(\bR)$,
\begin{equation}
\label{id1}
\int_{\bR}\int_{\bR}\varphi(x)\psi(y)|x-y|^{-(1-\alpha)}dxdy=\frac{1}{2\pi C_{1,\alpha} } \int_{\bR}\cF \varphi (\xi)\overline{\cF \psi(\xi)}|\xi|^{-\alpha}d\xi.
\end{equation}
Here $\cF \varphi(\xi)=\int_{\bR}e^{-i \xi x}\varphi(x)dx$ is the Fourier transform of $\varphi$.

For the sake of a comparison with some identities related the fractional Brownian motion (fBm), we use the parametrization $\alpha=2H-1$ with $H \in (\frac{1}{2},1)$. 
Then 
\begin{align*}
\frac{1}{2\pi C_{1,\alpha}}&= 
\pi^{-1/2} 2^{2H-2}\frac{\Gamma(H-\frac{1}{2})}{\Gamma(1-H)}=\pi^{-1/2} 2^{2H-2}\frac{\Gamma(H-\frac{1}{2})\Gamma(H)\sin(\pi H)}{\pi}\\
&= \frac{\sin(\pi H)}{\pi}\Gamma(2H-1)=\frac{\sin(\pi H)}{\pi}\cdot \frac{\Gamma(2H+1)}{2H(2H-1)}=\frac{c_H}{\alpha_H},
\end{align*}
using the Gamma identities
$\Gamma(1-z)\Gamma(z)=\frac{\pi}{\sin(\pi z)}$, $\Gamma(z) \Gamma(z+\frac{1}{2})=2^{1-2z}\sqrt{\pi}\Gamma(2z)$ and
$\Gamma(z+1)=z\Gamma(z)$ (in this order), and the notation:
\[
\alpha_{H}:=H(2H-1) \quad \mbox{and} \quad c_H:=\frac{\Gamma(2H+1)\sin(\pi H)}{2\pi}.
\]

In this parametrization, relation \eqref{id1} becomes: 
\begin{equation}
\label{id}
\alpha_H \int_{\bR} \int_{\bR}\varphi(x)\psi(y)|x-y|^{2H-2}dxdy=c_H \int_{\bR}\cF \varphi(\xi)\overline{\cF \psi(\xi)}|\xi|^{1-2H}d\xi,
\end{equation}
for any $\varphi,\psi \in \cS(\bR)$, and
\[
A_{\theta,w}=\frac{\alpha_H}{2\pi c_H}\int_{\bR}\int_{\bR}I_{\theta}(y) I_{w}(z)|y-z|^{2H-2}dydz.
\]

Identity \eqref{id} holds also for non-negative functions $\varphi,\psi \in L^1(\bR)$ for which $\cE(\varphi)<\infty$ and $\cE(\psi)<\infty$, where
\[
\cE(\varphi):=\alpha_H\int_{\bR}\int_{\bR} \varphi(x) \varphi(y)|x-y|^{2H-2}dxdy=c_H \int_{\bR}|\cF \varphi (\xi)|^2 |\xi|^{1-2H}d\xi.
\]
(Note that if $\varphi=1_{[0,t]}$ and $\psi=1_{[0,s]}$ for $t,s\in \bR$, then \eqref{id} coincides with the covariance $\bE[B_t B_s]$ of a fractional Brownian motion $(B_t)_{t\in \bR}$ of index $H$, where $1_{[0,t]}:=1_{[t,0]}$ if $t<0$.)

\medskip

In our case, we apply \eqref{id} to $I_{\theta},I_{w} \in L^1(\bR)$:
\[
A_{\theta,w}=\frac{\alpha_H}{2\pi c_H}\int_{\bR}\int_{\bR} I_{\theta}(y)I_{w}(z)|y-z|^{2H-2} dydz=\frac{1}{2\pi}\int_{\bR}\cF I_{\theta}(\xi)\overline{\cF I_{w}(\xi)}|\xi|^{1-2H}d\xi.
\]

We calculate the Fourier transform of $I_{\theta}$. By Fubini's theorem, for any $\xi \in \bR$,
\begin{align*}
\cF I_{\theta}(\xi)&=\int_{\bR}e^{-i\xi y}I_{\theta}(y)dy=\int_{\bR}e^{-i\xi y}\left(\int_{-\theta}^{\theta}G_t(x-y)dx\right)dy=\int_{-\theta}^{\theta} \cF G_t(x-\cdot)(\xi)dx\\
&=\overline{\cF G_t(\xi)} \int_{-\theta}^{\theta} e^{-i\xi x} dx=2\overline{\cF G_t(\xi)} \, \frac{\sin(|\xi|\theta)}{|\xi|}.
\end{align*}
It follows that
\begin{align*}
A_{\theta.w}&=\frac{2}{\pi}\int_{\bR}|\cF G_t(\xi)|^2 \frac{\sin(|\xi|\theta) \sin(|\xi|w)}{|\xi|^2}|\xi|^{1-2H}d\xi \leq \frac{2t^2}{\pi} \int_{\bR}\frac{\sin(|\xi|\theta) \sin(|\xi|w)}{|\xi|^2}|\xi|^{1-2H}d\xi,
\end{align*}
where for the last line we used the fact that $|\cF G_t(\xi)|\leq \int_{\bR}G_t(x)dx=t$.

By formula 2.6 (3) on p. 78 of \cite{erdelyi54},
\[
c_H \int_{\bR}\frac{\sin(|\xi|\theta) \sin(|\xi|w)}{|\xi|^2}|\xi|^{1-2H}d\xi=\frac{|\theta+w|^{2H}-|\theta-w|^{2H}}{4}.
\]
(This expression coincides with the covariance $\bE[B_w^o B_{\theta}^o]$, where $B_t^{o}=\frac{B_t-B_{-t}}{2}$ is ``odd part'' of the $(B_t)_{t\in \bR}$; see relation (10) of \cite{DZ04}.) Therefore,
\begin{align*}
A_{\theta,w} &\leq \frac{t^2}{2\pi c_H}  \big[ (w+\theta)^{2H}-(w-\theta)^{2H}\big]=\frac{t^2 H}{\pi c_H}  \int_{w-\theta}^{w+\theta}x^{2H-1}dx = \frac{t^2 H}{\pi c_H}   \int_{-\theta}^{\theta}(w+z)^{2H-1}dz.
\end{align*}

For any $z \in (-\theta,\theta)$, $w+z\geq w-\theta\geq 0$ and hence 
\[
(w+z)^{2H-1}=|w+z|^{2H-1}\leq w^{2H-1}+|z|^{2H-1}.
\]
It follows that
\begin{align*}
A_{\theta,w} & \leq  \frac{t^2 H}{\pi c_H}    \left( 2\theta w^{2H-1}+\int_{-\theta}^{\theta}|z|^{2H-1}dz\right)= \frac{t^2 H}{\pi c_H}   \left( 2\theta w^{2H-1}+\frac{1}{H}\theta^{2H}\right).
\end{align*}

By item (i) mentioned at the beginning of the proof,
$\sigma_{\theta} \sim K(t) \theta^{H}$ and $\sigma_{w} \sim K(t) w^{H}$. Hence,
\[
\frac{A_{\theta,w}}{\sigma_{\theta}\sigma_w} \leq C_{t,H} \left[ \left( \frac{\theta}{w}\right)^{1-H}+\left(\frac{\theta}{w}\right)^H\right],
\]
where $C_{t,H}>0$ is a constant that depends on $(t,H)$. Relation \eqref{claim} now follows by \eqref{H-A} and Lemma \ref{lemA} below.

This concludes the proof of Theorem \ref{main}.

\bigskip

\begin{remark}
\label{Gaussian}
{\rm Let $F_R(t)$ be the spatial average given by \eqref{def-F}, where $u(t,x)$ is the solution of (HAM) driven by the Gaussian noise $F$:
\[
F_t(\varphi)=W_t(\varphi* \kappa), \quad \mbox{for} \ t>0,\varphi \in \cS(\bR),
\]
and $W$ is the space-time Gaussian white noise on $\bR_{+}\times \bR$. Note that $F$ is the spatially-homogeneous Gaussian noise introduced in \cite{dalang99}, which is white in time and has covariance:
\[
\bE[F_t(\varphi)F_s(\psi)]=(t \wedge s)\int_{\bR}(\varphi* \kappa)(x)(\psi *\kappa)(x)dx= (t \wedge s)\int_{\bR^2}\varphi(x)\psi(y)\kerf(x-y)dxdy. 
\]

If $\kappa \in L^1(\bR)$ or $\kappa=R_{1,\alpha/2}$, properties {\bf (i)} and {\bf (ii)} still hold, as shown in \cite{DNZ20}.
Regarding property {\bf (iii)}, we notice that the first chaos term in the series expansion of $u(t,x)$ is:
\[
\int_0^t \int_{\bR}G_{t-r}(x-y)F(dr,dy)=\int_0^t \int_{\bR} \big(G_{t-r}(x-\cdot) * \kappa\big)(y)W(dr,dy).
\]

Therefore, $\big(G_{t-r}(x-\cdot) * \kappa\big)(y)$ is the first term (i.e. the mean) which appears in the series expansion of the Malliavin derivative $D_{r,y}^{W}u(t,x)$ of $u(t,x)$, {\em with respect to $W$}. (Note that $D_{r,y}^{W}u(t,x)$ is different than the usual Malliavin derivative $D_{r,y}u(t,x)$ of $u(t,x)$, with respect to $F$.) Using a well-established procedure (see e.g. \cite{DNZ20}), it can be shown that this first term dominates the other terms, in the sense that the following property holds:\\
{\bf (iii)'} for any $0\leq r\leq t$, $x \in \bR$, $y \in \bR$, and $p\geq 2$,
\[
\|D_{r,y}^Wu(t,x)\|_p \leq C_t' \int_{\bR}G_{t-r}(x-y')\kappa(y-y')dy',
\]
where $C_t'>0$ is a constant depending on $t$.

\medskip

To prove the ASCLT, we use the same procedure as above. Let $H_{\theta}$ be the same as above. By the Clark-Ocone formula with respect to $W$,
\[
H_{\theta}=\int_0^t \int_{\bR}\bE[D_{r,y}^W H_{\theta}|\cF_r]W(dr,dy),
\]
where $(\cF_t)_{t\geq 0}$ is the filtration induced by $W$. From this, we infer that:
\begin{equation}
\label{H-tw1}
|\bE[H_{\theta} H_{w}]| \leq \frac{{\rm Lip}^2(\lipf)}{\sigma_{\theta} \sigma_{w}} \int_0^t \int_{\bR} \|D_{r,y}^W F_{\theta}\|_2 \|D_{r,y}^W F_{w}\|_2 dydr.
\end{equation}

By Minkowski's inequality, property {\bf (iii)'}, and Fubini's theorem,
\begin{align*}
\|D_{r,y}^W F_{\theta}\|_2 & \leq \int_{-\theta}^{\theta}\|D_{r,y}^W u(t,x)\|_2 dx \leq C_t 
\int_{-\theta}^{\theta} \int_{\bR} G_{t-r}(x-y')\k(y-y')dy' dx\\
&= C_t  \int_{\bR}\varphi_{t,\theta}(r,y')\k(y-y')dy'.
\end{align*}

Inserting this estimate into \eqref{H-tw1}, we obtain:
\begin{align*}
& |\bE[H_{\theta}H_w]| \leq \\
& \quad \frac{{\rm Lip}^2(\lipf)}{\sigma_{\theta} \sigma_{w}} C_t^2  \int_0^t \int_{\bR} \left( \int_{\bR}\varphi_{t,\theta}(r,y')\k(y-y')dy' \right) 
\left( \int_{\bR}\varphi_{t,\theta}(r,y'')\k(y-y'')dy'' \right) dydr.
\end{align*}

Then, the same argument as above shows that if $\kappa \in L^1(\bR)$ or $\kappa=R_{1,\alpha/2}$, then the process $\{\widetilde{F}_{\theta}\}_{\theta>0}$ defined by \eqref{def-Ft} satisfies the ASCLT.
}
\end{remark}

\appendix

\section{An auxiliary result}

The following result is a version of Lemma 3.1 of \cite{BXZ}.

\begin{lemma}
\label{lemA}
Let $\{ H_\theta\}_{\theta > 0}$ be a family uniformly bounded random variables such that 
\[
\big| \bE[ H_{\theta} H_{w} ] \big| \leq C \left[ \left( \frac{\theta}{w} \right)^{\beta_1}+ \left( \frac{\theta}{w} \right)^{\beta_2}\right] \quad \mbox{for any $\theta<w$}, 
\]
for some $C,\beta_1,\beta_2>0$. Assume that $\theta\mapsto H_\theta$ is a measurable
function almost surely. Then,
\begin{align*} 
L_T:=\frac{1}{\log T} \int_1^T \frac{1}{\theta} H_\theta d\theta \stackrel{a.s.}{\longrightarrow} 0 \quad \mbox{as $T \to \infty$}.
\end{align*}
\end{lemma}

\begin{proof}
The second moment of $L_T$ is estimated as follows: 
\begin{align*}
\bE\big[ L_T^2 \big]
&= \frac{1}{(\log T)^2} \int_1^T \int_1^T \frac{1}{\theta w}  \bE[ H_\theta H_w ] \,d\theta dw =\frac{2}{(\log T)^2} \int_{1<\theta<w<T} \frac{1}{\theta w}  \bE[ H_\theta H_w ] \,d\theta \\
&\leq \frac{2C}{(\log T)^2} \int_{1<\theta<w<T}  \frac{1}{\theta w}
  \left[\left( \frac{\theta}{w} \right)^{\beta_1}+\left( \frac{\theta}{w} \right)^{\beta_2}\right] \,d\theta dw=\frac{2C}{\log T}\left( \frac{1}{\beta_1}+\frac{1}{\beta_2}\right)
 \end{align*}
using the fact that for any $\beta>0$,
\begin{align*}
\int_{1<\theta<w<T} \frac{1}{\theta w} \left( \frac{\theta}{w} \right)^{\beta} d\theta dw=\int_1^T \left(\int_1^w \theta^{\beta-1}d\theta \right)w^{-\beta-1}dw=\frac{1}{\beta}\int_1^T w^{\beta}w^{\beta-1}dw=\frac{\log T}{\beta}.
\end{align*}
The rest of the proof is the same as in Lemma 3.1 of \cite{BXZ}.

\end{proof}


\begin{thebibliography}{99}

\bibitem{B15} Balan, R.M. (2015). Integration with respect to L\'evy colored noise, with applications to SPDEs. {\em Stochastics} {\bf 87}, 363-381.
    

\bibitem{BJ25} Balan, R.M. and Jim\'enez, J.J. (2025). Moment estimates for solutions of SPDEs with L\'evy colored noise. {\em Stoch. PDEs}. To appear. 

\bibitem{BS26} Balan, R.M. and Stephenson, W.D. (2026). Gaussian fluctuations for hyperbolic Anderson model with L\'evy colored noise. Preprint available on arXiv:2602.23137.
    
\bibitem{BXZ} Balan, R.M., Xia, P. and Zheng, G. (2025). Almost sure central limit theorem for the hyperbolic Anderson model with L\'evy white noise. {\em Proc. Amer. Math. Soc.} {\bf 153}, 3083-3098. 
 
\bibitem{brosamler88} Brosamler, G.A. (1988). An almost everywhere central limit theorem.
{\em Math. Proc. Cambridge Philos. Soc.} {\bf 104}, 561-574.

   
\bibitem{dalang99} Dalang, R. C. (1999). Extending martingale
measure stochastic integral with applications to spatially
homogenous s.p.d.e.'s. {\em Electr. J. Probab.} {\bf 4}, no. 6,
1-29. 

\bibitem{DNZ20} Delgato-Vences, F., Nualart, D. and Zheng, G. (2020). A Central Limit Theorem for the stochastic wave equation with fractional noise. {\em Ann. Inst. Henri Poincar\'e: Prob. Stat.} {\bf 56}, 3020-3042.

\bibitem{DZ04} Dzhaparidze, K. and van Zanten, H. (2004). A series expansion of fractional Brownian motion. {\em Bernoulli} {\bf 130}, 39-55.

\bibitem{erdelyi54} Erd\'elyi, A., Magnus, W., Oberhettinger, F., Tricomi, F.G. (1954). {\em Tables of integral transforms. Vol. I.} McGraw-Hill, New York.
    
\bibitem{lacey-philipp90} Lacey, M.T. and Philipp, W. (1990).
A note on the almost sure central limit theorem. {\em Statist. Probab. Lett.} {\bf 9}, 201-205.

    
\bibitem{levy54} L\'evy, P. (1954). {\em Th\'eorie de l'addition des variables al\'eatoires}. Second edition. Gauthier-Villars, Paris.

\bibitem{schatte88} Schatte, P. (1988). On strong versions of the central limit theorem.
{\em Math. Nachr.} {\bf 137}, 249-256.


    
\bibitem{trauthwein25} Trauthwein, T. (2025). Quantitative CLTs on the Poisson space via Skorohod estimates and $p$-Poincar\'e inequalities. {\em Ann. Appl. Probab.} {\bf 35}, 1716-1754.
    
\bibitem{xia-zheng25} Xia, P. and Zheng, G. (2025). Almost sure central limit theorems for parabolic/hyperbolic Anderson models with Gaussian colored noises. {\em J. Theor. Probab.} {\bf 38}, article no. 46.

\end{thebibliography}
\end{document}